\documentclass[final,5p,times,twocolumn,numafflabel]{elsarticle}

\usepackage{amssymb}
\usepackage{amsmath}
\usepackage{amsthm}

\usepackage{comment}

\usepackage{graphicx}
\usepackage{hyperref}

\usepackage{xcolor}
\usepackage{float}
\usepackage{mathtools}
\usepackage{esint}
\usepackage{varioref}
\usepackage{comment}
\usepackage{enumitem}
\usepackage{ifthen}
\usepackage[capitalize,noabbrev]{cleveref}

\usepackage{cancel}

\usepackage[toc,page]{appendix}

\usepackage{mathrsfs}
\usepackage{marginnote}

\renewcommand{\d}{\mathrm{d}}
\renewcommand{\L}{\mathrm{L}}
\renewcommand{\H}{\mathrm{H}}

\newtheorem{theorem}{Theorem}

\newtheorem{lemma}{Lemma}
\newtheorem{definition}{Definition}

\newtheorem{remark}{Remark}
\newtheorem{proposition}{Proposition}

\newtheorem{assumption}{Assumption}

\def\Cc{\mathbb{C}}

\renewcommand{\d}{\mathrm{d}}
\renewcommand{\H}{\mathrm{H}}
\renewcommand{\L}{\mathrm{L}}

\newcommand{\re}{\mathrm{Re}\,}
\newcommand{\Pt}[1][t]{\mathbf{P}_{#1}}
\newcommand{\F}{\mathbb{F}}
\newcommand{\diag}{\mathrm{diag}}
\newcommand{\rank}{\mathrm{rank}}

\newcommand{\opt}{\mathrm{opt}}

\newcommand{\WBi}{\tilde W_{B,1}}
\newcommand{\WBh}{\tilde W_{B,2}}
\newcommand{\WCo}{\tilde W_C}

\newcommand{\zinf}[1][0]{[#1,\infty)}

\newcommand{\Lploc}[1][p]{\L_{\text{loc}}^{#1}}

\newcommand{\Hloc}[1][n]{\H_{\text{loc}}^{#1}}

\newcommand{\C}{\mathbb{C}}

\newcommand{\T}{\mathbb{T}}
\newcommand{\gs}{\sigma}
\newcommand{\Lin}{\mathcal{L}}
\newcommand{\citel}[2]{\cite[#2]{#1}}
\newcommand{\eq}[1]{\begin{align*}#1 \end{align*}}
\newcommand{\eqn}[1]{\begin{align}#1 \end{align}}
\newcommand{\pmat}[1]{\begin{bmatrix}#1 \end{bmatrix}}
\newcommand{\pmatsmall}[1]{\begin{bsmallmatrix}#1 \end{bsmallmatrix}}
\newcommand{\iprod}[2]{\langle #1,#2\rangle}
\newcommand*{\Iprod}[3][default]{\ifthenelse{\equal{#1}{default}}{\left\langle#2,#3\right\rangle}{\ldelim{#1}{\langle}#2,#3\rdelim{#1}{\rangle}}}
\newcommand{\norm}[1]{\| #1\|}

\newcommand{\Dom}{D}

\newcommand{\mc}[1]{\mathcal{#1}}
\newcommand{\ran}{\textup{Ran}}
\renewcommand{\ker}{\textup{Ker}}
\newcommand{\gl}{\lambda}
\newcommand{\inv}{^{-1}}

\journal{Elsevier}

\begin{document}

\begin{frontmatter}

\title{LQ optimal control for infinite-dimensional passive systems}

\author[label1]{Anthony Hastir}
\author[label1]{Birgit Jacob}

\affiliation[label1]{organization={School of Mathematics and Natural Sciences, University of Wuppertal},
             addressline={Gaußstraße 20},
             postcode={42119},
             city={Wuppertal},
             country={Germany}}



\begin{abstract}
We study the Linear-Quadratic optimal control problem for a general class of infinite-dimensional passive systems, allowing for unbounded input and output operators. We show that under mild assumptions, the finite cost condition is always satisfied. Moreover, we show that the optimal cost operator is a contraction. In the case where the system is energy preserving, the optimal cost operator is shown to be the identity, which allows to deduce easily the unique stabilizing optimal control. In this case, we derive an explicit solution to an adapted operator Riccati equation. We apply our results to boundary control systems, first-order port-Hamiltonian systems and an Euler-Bernoulli beam with shear force control.
\end{abstract}


\end{frontmatter}

\section{Introduction}\label{sec:Intro}
The Linear-Quadratic (LQ) optimal control problem has been widely studied over the past few years, for many different kinds of linear systems. Its solution can be easily characterized in the case where the dynamics is given by
\begin{subequations}
\label{eq:Finite_Dim}
\eqn{
  \dot{x}(t) &= Ax(t) + Bu(t), \qquad x(0) = x_0,\label{State_Finite}\\
  y(t) &= Cx(t) + Du(t),\label{Output_Finite}
  }
\end{subequations}

where $A$, $B$, $C$, $D$ are either matrices or $A$ is an (unbounded) operator that generates a strongly continuous semigroup on the Hilbert space $X$ and $B$, $C$, $D$ are bounded linear operators, i.e., $B\in\mathcal{L}(U,X)$, $C\in\mathcal{L}(X,Y)$, $D\in\mathcal{L}(U,Y)$, $U$ and $Y$ being the input and output (Hilbert) spaces. In both cases, the optimal control which minimizes the cost 
\begin{equation*}
J(x_0,u) = \int_0^\infty \Vert u(t)\Vert^2_U + \Vert y(t)\Vert^2_Y\d t
\end{equation*}
over $u\in\L^2(0,\infty;U)$ is given by the state feedback $u(t) = Kx(t) = -(I+D^*D)^{-1}(B^*\Theta+D^*C) x(t)$, where $\Theta = \Theta^*, \Theta\in\mathcal{L}(X)$, is the smallest nonnegative solution of the control algebraic Riccati equation
\begin{align}
&\langle Ax, \Theta x\rangle + \langle \Theta x, Ax\rangle + \langle Cx, Cx\rangle\nonumber\\
&= \langle (I + D^*D)^{-1}(B^*\Theta + D^*C)x,(B^*\Theta + D^*C)x\rangle,
\label{Riccati_WeakForm}
\end{align}
for $x\in D(A)$ or $x\in X$ if $X$ is finite-dimensional. The optimal cost is given by $\langle \Theta x_0,x_0\rangle = \inf_{u\in\L^2(0,\infty;U)}J(x_0,u)$. The theory of LQ optimal control for both finite and infinite-dimensional systems with bounded input and output operators is described extensively in \cite{Curtain_Pritchard_Riccati}, \cite[Chap.~4]{CurtainPritchard} and \cite{Kwakernaak_Sivan}, or in \cite[Chap.~9]{CurtainZwart2020}. 

In addition to a state space solution of LQ optimal control, a frequency domain approach can be used, known as the spectral factorization method. For\footnote{The notation $\rho(A)$ stands for the resolvent set of $A$.} $s\in\rho(A)$, we denote by $P(s)$ the transfer function of the system. Given the Popov function 
\begin{equation}
\Xi(i\omega) := I + P(i\omega)^*P(i\omega),\quad  \omega\in\mathbb{R},
\label{eq:Popov}
\end{equation}
the spectral factorization method consists in finding a spectral factor of $\Xi$, that is, a function\footnote{The Hardy space $\mathbf{H}^\infty(\mathcal{L}(U))$ consists in analytic functions $P:\mathbb{C}_+\to\mathcal{L}(U)$ that are bounded in $\mathbb{C}_+$, i.e., $\sup_{s\in\mathbb{C}_+}\Vert P(s)\Vert<\infty$, with $\mathbb{C}_+ = \{s\in\mathbb{C}\ | \  \mathrm{Re}\,s> 0\}$.} $\chi\in \mathbf{H}^\infty(\mathcal{L}(U))$ for which $\chi^{-1}\in \mathbf{H}^\infty(\mathcal{L}(U))$ and that satisfies 
\begin{align*}
\Xi(i\omega) = \chi(i\omega)^*\chi(i\omega), 
\end{align*}
for almost every $\omega\in\mathbb{R}$. The inverse of the so-determined spectral factor is precisely the closed-loop transfer function up to a scaling factor. This can be used to determine the optimal feedback operator. This method has been extensively studied in \cite{CallierWinkin1990} and \cite{CallierWinkin1992} in the case where the input and the output operators are bounded. 

The approaches described above are no more valid when the input and output operators are allowed to be unbounded. In such a situation, 
the authors in \cite{WeissWeiss} and \cite{Staffans_LQ_1998} proposed the following operator Riccati equation
\begin{align}
  &\langle Ax, \Theta x\rangle + \langle \Theta x, Ax\rangle + \langle Cx, Cx\rangle\nonumber\\
  &= \langle (\Omega^*\Omega)^{-1}(B^*_\omega\Theta + D^*C)x,(B^*_\omega\Theta + D^*C)x\rangle,
\label{Riccati_Unbounded}
\end{align}
$x\in D(A)$, where $B^*_\omega$ is the Yosida extension of $B^*$, whose precise definition is given in e.g.~\cite{Weiss_Yosida_Ext} or \cite{Tucsnak_Weiss_LTI_Case}. In the above Riccati equation, the operator $\Omega$ is obtained by taking the limit of $\chi(s)$ for $s$ going to $\infty$, $s$ real, provided that the limit exists, that is, provided that the spectral factor is regular. The relation $I + D^*D = \Omega^*\Omega$ is in general not satisfied, unlike in the case where the input and output operators are bounded. Another major difference between the bounded and the unbounded situations is that the closed-loop system is not necessarily obtained by connecting the optimal feedback to the original system, see~\cite[Sec.~3]{Weiss_Zwart}. The spectral factorization method described above has also been extended to systems with unbounded input and output operators in \cite{WeissWeiss} where the authors consider weakly regular systems. Further developments on this method may be found in \cite{Opmeer_Staffans} and \cite{Opmeer_MTNS}.  

Different variations of the Riccati equation \eqref{Riccati_Unbounded} exist, see e.g.~\cite{Pritchard_Salamon}, \cite{WeissWeiss}, \cite{Curtain_LQ_SIAM} and \cite{mikkola2016riccati}. In \cite{Curtain_LQ_SIAM}, Riccati equations for well-posed linear systems have been considered under the assumption that $0$ is in the resolvent set of the operator dynamics, giving the possibility to define a reciprocal Riccati equation, based on reciprocal systems. The reciprocal system has the advantage of only involving bounded operators but those operators are often difficult to compute. A few years later, \cite{mikkola2016riccati} proposed generalizations of previous work on LQ optimal control by introducing an integral Riccati equation. More recently, another type of Riccati equation called the \textit{operator node Riccati equation} has been considered in \cite{Opmeer_MTNS} and \cite{Opmeer_Staffans} for general infinite-dimensional systems. 

In the particular situation where $A = -A^*$, $B = C^*$ and $D = 0$ in \eqref{eq:Finite_Dim}, it is easy to see that $\Theta = I$ is a nonnegative solution to the Riccati equation \eqref{Riccati_WeakForm}, see \cite{Zwart1996}. Moreover, if the negative output feedback $u(t) = -y(t)$ is strongly stabilizing (that is, the semigroup generated by $A-BB^*$ is strongly stable), then $\Theta = I$ is the unique nonnegative stabilizing solution to the Riccati equation \eqref{Riccati_WeakForm}. Note that the strong stability of the semigroup generated by $A-BB^*$ may be obtained by assuming that $A$ has compact resolvent and that $(A,B,B^*,0)$ is approximately controllable, see e.g. \citel{CurZwa16}{Thm.~1.2} and \citel{Sle89}{Sec.~2}. Systems $(A,B,C,0)$ with $A$ being a dissipative operator and with $B = C^*$ are more commonly called passive systems. Roughly speaking, such systems have the property that
the energy supplied to the system is bounded by the product between the input and the output, that is,
\eq{
\Vert x(t)\Vert_X^2 - \Vert x_0\Vert^2_X\leq 2\int_0^t\re\langle u(\tau),y(\tau)\rangle_U\d\tau,
}
for every $t\geq 0$. If in addition $A = -A^*$, then the system is called impedance energy preserving, which means that the supplied energy is exactly the product between inputs and outputs, i.e., the inequality above is an equality. In particular, the relation $B = C^*$ means that the input and the output are collocated.  Another kind of passivity is called \emph{scattering passivity}. A system that possesses this property satisfies
\eq{
\Vert x(t)\Vert_X^2-\Vert x_0\Vert_X^2 \leq \Vert u\Vert^2_{\L^2(0,t;U)}-\Vert y\Vert_{\L^2(0,t;Y)}^2,
}
for every $t\geq 0$, and an equality in the above estimate means that the system is scattering energy preserving.
In this manuscript, we focus on passive systems with impedance and dispersion and consider fairly general systems that allow for unbounded input and output operators. This leads us to the class of so-called \emph{well-posed linear systems}, see \citel{TucWei14}{Def.~3.1} for more details. 

Passive systems have been studied a lot in the last decades. Passivity has been characterized in the context of well-posed linear systems in \cite{Sta02}, where the author uses the formalism of system nodes. Few years later, the authors of \cite{MalSta06} studied conservative boundary control systems, while impedance passive and conservative boundary control systems are extensively analyzed in \cite{MalSta07}. Passive systems also attracted a lot of attention from an application point of view. Many examples of passive systems can be found in the survey \cite{TucWei14}. Port-Hamiltonian systems are also known to be passive; see, e.g. \cite{Maschke_1992} or \cite{Schaft1995}. Infinite-dimensional port-Hamiltonian systems have also been studied in \cite{JacobZwart} where it is mentioned that some properties of the system matrices imply passivity. 

In this manuscript, we study the LQ optimal control problem in the context of passive systems. We show that, under some mild assumptions, the optimal cost operator has surprisingly nice properties. For energy preserving systems we derive an explicit solution for the optimal cost operator, which leads to a unique optimal control that is either a negative output feedback for impedance energy preserving systems or the zero input in the case of scattering energy preserving systems. A similar approach is used to compute the optimal control for a class of second-order systems with collocated input and outputs in \cite{Wei02}. Therein, the author transforms the original system into a scattering energy preserving system in order to compute the unique optimal control. Our approach is based on the formalism of system nodes, which allows for a relatively high degree of generality in the considered types of systems, including for instance systems with boundary control and boundary observation, see~\cite{Sta05book}. Our results on LQ optimal control are based on the so-called \emph{operator-node Riccati equation} proposed in \cite{Opmeer_Staffans} and on the characterization of passivity for systems nodes, see \cite{Sta02}. We emphasize that this approach has the advantage of not requiring any extension of unbounded operator involved in the Riccati equation \eqref{Riccati_Unbounded}. We also take advantage of our results to provide a Riccati equation that is well-adapted to boundary control systems, allowing to apply our results to the class of first-order linear port-Hamiltonian systems considered in \cite{JacobZwart}. This manuscript is organized as follows: Section \ref{sec:Prelim} is dedicated to the presentation of our systems class. Our main results on LQ optimal control for passive systems are presented in Section \ref{sec:LQ_SystemNodes}. Those results are recast for linear boundary control systems in Section \ref{sec:LQ_BCS} and for a class of first-order port-Hamiltonian systems in Section \ref{sec:App_pHs}. We illustrate our main result on an Euler-Bernoulli beam with shear force control in Section \ref{sec:EB}.

\subsection*{Notation} Given $X$ a Hilbert space and $A:\Dom(A)\subset X\rightarrow X$ a linear operator we denote by $\Dom(A)$, $\ker(A)$ and $\ran(A)$ the domain, kernel, and range of $A$, respectively. Moreover, $\gs(A)$
and $\rho(A)$ denote the spectrum
and the \mbox{resolvent} set of $A$, respectively. Given another Hilbert space $Y$, $\Lin(X,Y)$ denotes the space of bounded linear operators from $X$ to $Y$ and we write $\Lin(X)$ for $\Lin(X,X)$. For $\tau>0$ and $u:\zinf\to U$ the truncation of the function $u$ to the interval $[0,\tau]$ is defined by $\Pt[\tau]u: [0,\tau]\to U$ and $ \mathbf{S}_\tau $ stands for the bilateral right shift operator. For any $ u, v\in \L^2(0,\infty;U) $ and any $\tau\geq 0 $, the $ \tau-$concatenation of $ u $ and $ v $ is the function defined by $u\underset{\tau}{\diamondsuit} v = \mathbf{P}_{[0,\tau]} u + \mathbf{S}_\tau v$. The notation $\Cc_+$ is used for $\{\gl\in \Cc \ | \ \re\gl>0\}$. For a generator $A$ of a strongly continuous semigroup we denote by $X_{-1}$ the completion of the space $(X,\norm{(\gl_0-A)\inv \cdot}_X)$ where  $\gl_0\in\rho(A)$ is fixed. According to ~\citel{Sta05book}{Sec.~3.6}, $A:\Dom(A)\subset X\to X$ extends to an operator in $\Lin(X,X_{-1})$, that we will denote by $A$ as well. The Lebesgue space of measurable and square integrable functions defined on $\Omega$ with values in a Banach space $W$ is denoted by $\L^2(\Omega;W)$. This is a Hilbert space with norm $\Vert f\Vert^2_{\L^2(\Omega;W)} := \int_\Omega \Vert f(\omega)\Vert^2_W\d\omega$. A function $f$ is said to be in $\L^2_{\mathrm{loc}}(\Omega;W)$ if $f\in\L^2(K;W)$ for every compact set $K\subseteq \Omega$. The space of 
essentially bounded functions defined on $\Omega$ with values in $W$ is denoted by $\L^\infty(\Omega;W)$. The norm on $\L^\infty(\Omega;W)$ is defined by $\Vert f\Vert_{\L^\infty(\Omega;W)} := \sup_{\omega\in\Omega}\Vert f(\omega)\Vert_W$. The space of continuous functions defined on $\Omega$ with values in $W$ is written $C(\Omega;W)$ whereas a continuous function whose first derivative is again continuous is said to belong to $C^1(\Omega;W)$. The Sobolev space of square integrable functions from $\Omega$ to $W$ whose weak derivatives up to order $k$ are square integrable is denoted by $\H^k(\Omega;W)$. The $\H^k(\Omega;W)$-norm of a function $f$ is defined as $\Vert f\Vert_{\H^k(\Omega;W)}^2 := \sum_{n=0}^k\Vert \frac{\d^n f}{\d\omega^n}\Vert_{\L^2(\Omega;W)}^2$. A function $f$ is said to be in $\H^k_{\mathrm{loc}}(\Omega;W)$ if it is in $\H^k(K;W)$ for every compact set $K\subseteq \Omega$. The Hardy space $\mathbf{H}^\infty(\mathcal{L}(U))$ consists in analytic functions $P:\mathbb{C}_+\to\mathcal{L}(U)$ that are bounded in $\mathbb{C}_+$, that is, $\sup_{s\in\mathbb{C}_+}\Vert P(s)\Vert<\infty$, where $\mathbb{C}_+ = \{s\in\mathbb{C}\ | \  \mathrm{Re}\,s> 0\}$.

\section{Well-posed linear systems and system nodes}\label{sec:Prelim}
In this section, we review some results and properties on well-posed linear systems and system nodes with a particular focus on passivity.
\subsection{Well-posed linear systems}
Let $U$, $X$ and $Y$ be Hilbert spaces. The concept of a well-posed linear system is defined hereafter, see~\citel{TucWei14}{Def.~3.1}.
\begin{definition}\label{def:WP}
A well-posed linear system is a family of operators $\Sigma = (\Sigma_t)_{t\ge 0} = (\T_t,\Phi_t,\Psi_t,\F_t)_{t\ge0}$ where 
\begin{itemize}
\item $ \mathbb{T} = \left(\mathbb{T}_t\right)_{t\geq 0} $ is a strongly continuous semigroup on $ X $,
\item $ \Phi = \left(\Phi_t\right)_{t\geq 0} $ is a family of bounded linear operators from $ \L^2(0,\infty;U) $ to $ X $ such that $\Phi_{\tau + t}(u\underset{\tau}{\diamondsuit}v)  = \mathbb{T}_t \Phi_\tau u + \Phi_t v$, for every $ u,v\in \L^2(0,\infty;U) $ and all $ \tau,t\geq 0 $,
\item $ \Psi = \left(\Psi_t\right)_{t\geq 0} $ is a family of bounded linear operators from $ X $ to $ \L^2(0,\infty;Y) $ such that $\Psi_{\tau + t}x_0 = \Psi_\tau x_0 \underset{\tau}{\diamondsuit}\Psi_t \mathbb{T}_\tau x_0$, for every $ x_0\in X $ and all $ \tau,t\geq 0 $ and $ \Psi_0 = 0 $,
\item $ \mathbb{F} = \left(\mathbb{F}_t\right)_{t\geq 0} $ is a family of bounded linear operators from $ \L^2(0,\infty;U) $ to $ \L^2(0,\infty;Y) $ such that
$\mathbb{F}_{\tau + t}(u\underset{\tau}{\diamondsuit} v) = \mathbb{F}_\tau \underset{\tau}{\diamondsuit}(\Psi_t \Phi_\tau u + \mathbb{F}_t v)$, for every $ u,v\in \L^2(0,\infty;U) $ and all $ \tau,t\geq 0 $, and $ \mathbb{F}_0 = 0$,
\item $ \Phi $ is causal, i.e., $ \Phi_\tau \mathbf{P}_\tau = \Phi_\tau $ for all $ \tau \geq 0 $,
\item For all $ \tau, t\geq 0 $, $\Phi_{\tau + t}\mathbf{P}_{[0,\tau]} = \mathbb{T}_t \Phi_\tau$, $\mathbf{P}_{[0,\tau]} \Psi_{\tau + t} = \Psi_\tau$, $\mathbf{P}_{[0,\tau]}\mathbb{F}_{\tau + t}\mathbf{P}_{[0,\tau]} = \mathbf{P}_{[0,\tau]} \mathbb{F}_{\tau + t} = \mathbb{F}_\tau$.
\end{itemize}
\end{definition}
In the above definition, $U$, $X$ and $Y$ are called the input space, the state space and output space, respectively.
We define the \emph{mild state trajectory} $x\in C(\zinf;X)$ and \emph{mild output} $y\in\Lploc[2](0,\infty;Y)$
corresponding to the initial state 
$x_0\in X$  and the input $u\in \Lploc[2](0,\infty;U)$  so that
\begin{subequations}
\label{eq:PrelimMildStateOut}
\eqn{
x(t) &= \mathbb{T}_t x_0 + \Phi_t\Pt u,\label{Eq_State}\\ 
\mathbf{P}_t y &= \Psi_t x_0 + \mathbb{F}_t\Pt u,\label{Eq_Output} 
}
\end{subequations}
for all $t\geq 0$.
The \emph{extended output map} and \emph{extended input-output map}~\citel{TucWei14}{Sec.~3} of $\Sigma$ are denoted by $\Psi_\infty:X\to \Lploc[2](0,\infty;Y)$ and $\F_\infty:\Lploc[2](0,\infty;U)\to \Lploc[2](0,\infty;Y)$, respectively.
Using these operators the output $y$ corresponding to the initial state $x_0$ and input $u$ can be equivalently expressed as $y=\Psi_\infty x_0 + \F_\infty u$.
In the next definitions, we describe passivity for well-posed linear systems. We distinguish between different kinds of passive systems, namely, \emph{impedance passive}, \emph{impedance energy preserving}, \emph{scattering passive} and \emph{scattering energy preserving} systems.
\begin{definition}\label{def:Imp_Passive}
A well-posed system $\Sigma$ is called \emph{impedance passive} if $Y=U$ and
 if every mild state trajectory $x$ and mild output $y$ corresponding to 
an initial state $x_0\in X$ and an input $u\in \Lploc[2](0,\infty;U)$ satisfy
\eqn{
\label{eq:ImpPass_Sigma}
\Vert x(t)\Vert_X^2 - \Vert x_0\Vert_X^2
\leq 2\re\int_0^t\langle u(\tau),y(\tau)\rangle_U\d\tau, \qquad t\geq 0.
}
$\Sigma$ is called impedance energy preserving if the inequality in \eqref{eq:ImpPass_Sigma} is an equality.
\end{definition}

\begin{definition}\label{def:Scatt_Passive}
A well-posed system $\Sigma$ is called \emph{scattering passive}
if every mild state trajectory $x$ and mild output $y$ corresponding to 
an initial state $x_0\in X$ and an input $u\in \Lploc[2](0,\infty;U)$ satisfy
\eqn{
\label{eq:ScattPass_Sigma}
\Vert x(t)\Vert_X^2 - \Vert x_0\Vert_X^2
\leq \Vert u\Vert_{\L^2(0,t;U)}^2 - \Vert y\Vert_{\L^2(0,t;Y)}^2, \qquad t\geq 0.
}
$\Sigma$ is called scattering energy preserving if the inequality in \eqref{eq:ScattPass_Sigma} is an equality.
\end{definition}

\subsection{System nodes}\label{sec:SysNodes}
The concept of a \emph{system node} is closely related to well-posed linear systems and is defined in the following definition. 

\begin{definition}[\textup{\citel{Sta05book}{Def.~4.7.2}}]
\label{def:SysNode} 
Let $X$, $U$, and $Y$ be Hilbert spaces.
A closed operator 
\eq{
S:=\pmat{A\&B\\ C\&D} : \Dom(S) \subset X\times U \to X\times Y
}
is called a \emph{system node} on the spaces $(U,X,Y)$ if it has the following properties.
\begin{itemize}
\item The operator $A: \Dom(A)\subset X\to X$ defined by $Ax =A\&B \pmatsmall{x\\0}$ for $x\in \Dom(A)=\{x\in X\ | \ (x,0)^\top\in \Dom(S)\}$ generates a strongly continuous semigroup on $X$.
\item The operator $A\& B$ (with domain $\Dom(S)$) can be extended to an operator $[A,\ B]\in \Lin(X\times U,X_{-1})$.
\item $\Dom(S) = \{(x,u)^\top \in X\times U\ | \ Ax+Bu\in X\}$.
\end{itemize}
\end{definition}

The operator $A$ is called the semigroup generator of $S$.
Every well-posed linear system $\Sigma = (\T,\Phi,\Psi,\F)$ is associated with a unique system node $S$~\citel{TucWei14}{Sec.~4}.
In particular, if $x_0\in X$ and $u\in \Hloc[1](0,\infty;U)$ are such that $(x_0,u(0))^\top \in \Dom(S)$, then the mild state trajectory $x$ and output $y$ of $\Sigma$ defined by~\eqref{eq:PrelimMildStateOut} satisfy $x\in C^1(\zinf;X)$, $x(0)=x_0$ and $y\in \Hloc[1](0,\infty;Y)$ and
\eqn{
\label{eq:PrelimSysnodeEqn}
\pmat{\dot x(t)\\y(t)} = S \pmat{x(t)\\ u(t)}, \qquad t\geq 0.
}
A system node $S$ is called well-posed if it is the system node of some well-posed linear system $\Sigma$. The conditions in Definition~\ref{def:SysNode} imply that $C\& D\in \Lin(\Dom(S),Y)$. 
The transfer function 
$P:\rho(A)\to \Lin(U,Y)$ of a system node $S$ on $(U,X,Y)$ is defined
by
\eq{
P(s)u = C\& D \pmat{(s-A)\inv Bu\\ u}, \qquad u\in U, \ s\in\rho(A).
}
In addition, the output operator $C\in \Lin(\Dom(A),Y)$ of $S$ is defined by $Cx= C\& D \pmatsmall{x\\0}$ for all $x\in \Dom(A)$. Before continuing, we introduce the concepts of classical and generalized solutions for systems nodes, see~\citel{TucWei14}{Def.~4.2}.
\begin{definition}\label{def:classicalSol_SN}
Let $S$ be a system node on the spaces $(U,X,Y)$ with domain $D(S)$. A triple $(x,u,y)$ is called a classical solution of \eqref{eq:PrelimSysnodeEqn} on $[0,\infty)$ if 
\begin{itemize}
    \item $x\in C^1([0,\infty);X)$, $u\in C([0,\infty);U)$ and $y\in C([0,\infty);Y)$;
    \item $\pmatsmall{x(t)\\u(t)}\in D(S)$ for all $t\geq 0$;
    \item \eqref{eq:PrelimSysnodeEqn} holds.
\end{itemize}
\end{definition}
According to \citel{TucWei14}{Sec.~4}, it follows from Definition \ref{def:classicalSol_SN} that every classical solution of \eqref{eq:PrelimSysnodeEqn} on $[0,\infty)$ also satisfies $\pmatsmall{x\\u}\in C([0,\infty);D(S))$.
\begin{definition}\label{def:genSol_SN}
Let $S$ be a system node on the spaces $(U,X,Y)$ with domain $D(S)$. A triple $(x,u,y)$ is called a generalized solution of \eqref{eq:PrelimSysnodeEqn} on  $[0,\infty)$ if
\begin{itemize}
    \item $x\in C([0,\infty);X)$, $u\in\L^2_{\mathrm{loc}}(0,\infty;U)$, $y\in\L^2_{\mathrm{loc}}(0,\infty;Y)$;
    \item there exists a sequence $(x_k,u_k,y_k)$ of classical solutions of \eqref{eq:PrelimSysnodeEqn} such that $x_k\to x$ in $C([0,\infty);X)$, $u_k\to u$ in $\L^2_{\mathrm{loc}}(0,\infty;U)$ and $y_k\to y$ in $\L^2_{\mathrm{loc}}(0,\infty;Y)$.
\end{itemize}
By $u_k\to u$ in $\L^2_{\mathrm{loc}}(0,\infty;U)$, we mean $\mathbf{P}_{[0,\tau]}u_k\to\mathbf{P}_{[0,\tau]}u$ in $\L^2(0,\tau;U)$ for every $\tau\geq 0$.
\end{definition}
We will now define \emph{passivity} of a system node, based on algebraic conditions on the operators $A\&B$ and $C\&D$, see \citel{Sta02}{Thm.~4.2, Thm.~3.3}.
\begin{definition}\label{def:ImpPass}
A system node $S$ on $(U,X,Y)$ is \emph{impedance passive} if $Y=U$ and 
\eqn{
    \re \left\langle A\& B
        \pmat{x\\ u}
,x\right\rangle_X \leq \re \left\langle C\& D\pmat{x\\ u },u\right\rangle_U
\label{eq:EstimEquivImpPass}
}
for all $(x, u)^\top\in \Dom(S)$. $S$ is impedance energy preserving if the inequality in \eqref{eq:EstimEquivImpPass} is an equality.
\end{definition}

\begin{definition}\label{def:ScattPass}
A system node $S$ on $(U,X,Y)$ is \emph{scattering passive} if 
\eqn{
    2\re \left\langle A\& B
        \pmat{x\\ u}
,x\right\rangle_X \leq \Vert u\Vert_U^2 - \left\Vert C\& D\pmat{x\\ u }\right\Vert_Y^2
\label{eq:EstimEquivScattPass}
}
for all $(x, u)^\top\in \Dom(S)$. $S$ is scattering energy preserving if the inequality in \eqref{eq:EstimEquivScattPass} is an equality.
\end{definition}

In the literature, it is more common to use a definition based on the solutions of~\eqref{eq:PrelimSysnodeEqn}, but the definitions are equivalent by~\citel{Sta02}{Thm.~3.3, Thm.~3.4, Thm.~4.2, Thm.~4.6}. In this way, a system node $S$ is impedance (resp. scattering) passive if all the generalized solutions of \eqref{eq:PrelimSysnodeEqn} satisfy \eqref{eq:ImpPass_Sigma} (resp. \eqref{eq:ScattPass_Sigma}) for every $t\geq 0$. Due to this equivalence, a well-posed system node $S$ is impedance (resp. scattering) passive if and only if its associated well-posed linear system $\Sigma$ is impedance (resp. scattering) passive. The same characterization is valid for impedance (resp. scattering) energy preserving systems.

\begin{remark}
The transfer function of an impedance passive well-posed system node $S$ coincides with the transfer function of the associated well-posed linear system $\Sigma$ on $\Cc_+$.
Moreover, if $S$ is an impedance passive system node, it has been shown in \citel{CurWei19}{Thm.~4.2} that the operator $-K\in\mathcal{L}(U)$ is an admissible output feedback operator for $S$. In addition, according to \citel{HasPau25}{Thm.~2.5}, if $\re K\geq I$, then the negative output feedback $-K$ applied to $S$ leads to a scattering passive system node, and hence well-posed.
\end{remark}

\section{LQ optimal control of passive systems}\label{sec:LQ_SystemNodes}
This section is devoted to the LQ optimal control for passive systems by using the formalism of system nodes. It will be shown that, under mild conditions, the optimal cost operator has interesting properties. This is done by taking advantage of the so-called \emph{operator-node} Riccati equation proposed in \cite{Opmeer_Staffans} and of the characterizations of passivity for system nodes described in Section \ref{sec:SysNodes}. We also provide a simple factorization of the so-called \emph{Popov} function in the case where the system node is impedance or scattering energy preserving. 

Mathematically speaking, the control objective consists in minimizing the cost 
\eqn{
\label{eq:Cost}
J(x_0,u) = \int_0^\infty \Vert u(t)\Vert^2_U + \Vert y(t)\Vert_Y^2\d t
}
subject to the dynamics 
\eqn{
\label{eq:SN}
\pmat{\dot{x}(t)\\y(t)} = \pmat{A\& B\\ C\&D}\pmat{x(t)\\u(t)}, \qquad t\geq 0,
}
where $S := \pmatsmall{A\&B\\C\&D}$ is a system node on the spaces $(U,X,Y)$. LQ optimal control and system nodes have been combined in \cite{Opmeer_Staffans} in which the authors provide an \emph{operator-node} Riccati equation that leads to the optimal cost operator. In \cite{Opmeer_MTNS}, it has also been shown that this Riccati equation is general and comprises other descriptions of Riccati equations used to treat the LQ optimal control problem for systems with unbounded input and output operators. Under an optimizability condition\footnote{Optimizability means that there exists at least one input $u(t)$ such that $J(x_0,u)<\infty$ for every $x_0\in X$. This is often called the \emph{finite cost} condition.}, it is known that minimizing the cost \eqref{eq:Cost} subject to \eqref{eq:SN} leads to a unique minimum. Moreover, the optimal cost is given by $J(x_0,u^{\opt}) = \langle x_0,\Pi x_0\rangle_X$, where $\Pi = \Pi^*\in\mathcal{L}(X)$ is the solution to an appropriate operator Riccati equation, namely the \emph{operator node} Riccati equation. Its definition is given below.
\begin{definition}
Let $S$ be a system node on the spaces $(U,X,Y)$. The operators $\Pi = \Pi^*\in\mathcal{L}(X)$ and $E\& F: D(S)\to U$ are called a solution to the operator node Riccati equation for $S$ if
\begin{align}
&\left\langle A\& B\left[\begin{matrix}x\\ u\end{matrix}\right], \Pi x\right\rangle_X + \left\langle \Pi x, A\& B\left[\begin{matrix}x\\ u\end{matrix}\right]\right\rangle_X + \left\Vert C\& D\left[\begin{matrix}x\\ u\end{matrix}\right]\right\Vert^2_Y + \left\Vert u\right\Vert^2_U\nonumber\\
&\hspace{5cm} = \left\Vert E\& F\left[\begin{matrix}x\\ u\end{matrix}\right]\right\Vert^2_U
\label{ON_Riccati}
\end{align}
for all $(x,u)^\top\in D(S)$. 
\end{definition}

This operator node Riccati equation might be surprising at first glance because no apparent quadratic term seems to be present. This is precisely the role of the operator $E\&F$, which encodes that the optimal input is of feedback form. According to \cite{Opmeer_MTNS}, solving the equation $E\&F\left[\begin{smallmatrix}x^{\opt}\\u^{\opt}\end{smallmatrix}\right] = 0$ leads to the optimal input $u^{\opt}$. The operator node Riccati equation \eqref{ON_Riccati} may be viewed as the extension of the Lur'e form of the Riccati equation to system nodes. Compared to other Riccati equations, see e.g.~\cite{WeissWeiss, Curtain_LQ_SIAM, Curtain_Pritchard_Riccati}, the operator node Riccati equation has the advantage of not having any extension of unbounded operators in it. Moreover, it even does not require well-posedness of the system node $S$.

\subsection{Impedance passive systems}

Consider a system node $S := \pmatsmall{A\&B\\ C\&D}$ on the spaces $(U,X,U)$. Moreover, consider the system node $S^K:= \left[\begin{smallmatrix}(A\&B)^K\\ (C\&D)^K\end{smallmatrix}\right]$ obtained by interconnecting $S$ with the negative output feedback $u = -y+v$. According to \citel{HasPau25}{Thm.~2.5}, $S^K$ is a scattering passive\footnote{As a consequence, $S^K$ is well-posed.} system node whose semigroup generator, denoted by $A^K$, is given by 
\begin{subequations}
\label{eq:StabAK}
\eqn{
\Dom(A^K) &= \Bigl\{ x\in X\, \Bigm| \exists v\in U \mbox{ s.t.}
\ Ax - Bv\in X,\
v = C\& D \pmatsmall{x\\-v}
\Bigr\}\\
A^K x &= Ax - Bv(x),
\qquad x\in \Dom(A^K),
}
\end{subequations}
where $v(x)$ is the element $v$ in the definition of $\Dom(A^K)$. In what follows, we denote by $\mathbb{T}^K$ the semigroup generated by $A^K$.

Recall that a strongly continuous semigroup $\T$ on $X$ is called \emph{strongly stable} if $\norm{\T_tx_0}\to 0$ as $t\to \infty$ for all $x_0\in X$.


The following theorem gives some properties of the optimal cost operator when $S$ is impedance passive or impedance energy preserving.

\begin{theorem}\label{thm:Thm_LQ_ImpPass}
Consider $S = \left[\begin{smallmatrix}A\& B\\ C\& D\end{smallmatrix}\right]$ an impedance passive system node on the spaces $(U,X,U)$ and consider the cost \eqref{eq:Cost} associated to the LQ optimal control problem. Then the finite cost condition is satisfied and the optimal cost operator, denoted by $\Pi$, satisfies $\Pi\leq I$ in the sense that $\langle x_0,\Pi x_0\rangle\leq\Vert x_0\Vert^2$ for all $x_0\in X$. If in addition $S$ is impedance energy preserving and if $\mathbb{T}^K$ is strongly stable, then the unique stabilizing optimal control is $u(t) = -y(t)$, $t\geq 0$. The associated optimal cost operator is $\Pi = I$. In this case, $\Pi$ satisfies the operator-node Riccati equation with $E\&F:D(S)\to U$,  $E\&F\left[\begin{smallmatrix}x\\ u\end{smallmatrix}\right] = C\&D\left[\begin{smallmatrix}x\\ u\end{smallmatrix}\right] + u$.
\end{theorem}

\begin{proof}
We start by considering the particular input $u(t) = -y(t)$ that leads to the well-posed system node $S^K$ with $\mathbb{T}^K$ as associated semigroup. Thanks to the impedance passivity of $S$ and to the well-posedness of $S^K$, it holds that
\begin{align*}
    \int_0^t\Vert u(\tau)\Vert^2_U + \Vert y(\tau)\Vert_Y^2\d\tau &= \int_0^t -2\re\langle u(\tau),y(\tau)\rangle_U\d\tau\\
    &\leq \Vert x_0\Vert_X^2 - \Vert \mathbb{T}^K_tx_0\Vert_X^2\\
    &\leq \Vert x_0\Vert_X^2.
\end{align*}
Taking the limit for $t$ going to $\infty$ implies that
\begin{align*}
    J(x_0,u) = \lim_{t\to\infty}\int_0^t\Vert u(\tau)\Vert^2_U + \Vert y(\tau)\Vert_Y^2\d \tau\leq \Vert x_0\Vert_X^2,
\end{align*}
which shows that $u$ and $y$ are elements of $\L^2(0,\infty;U)$ and that the finite cost condition is satisfied. Moreover, this shows that $\langle x_0,\Pi x_0\rangle = \inf_{u\in\L^2(0,\infty;U)}J(x_0,u)\leq \Vert x_0\Vert_X^2$ for all $x_0\in X$, which implies $\Pi\leq I$. Now we assume that $S$ is impedance energy preserving and that $\mathbb{T}^K$ is strongly stable. Moreover, we consider inputs $u$ for which $(x,u,y)$ is a generalized solution of $S$. Along those generalized solutions, it holds that 
\begin{align*}
    &\int_0^t\Vert u(\tau)\Vert^2_U + \Vert y(\tau)\Vert^2_U\d\tau\\
    &= \int_0^t \Vert u(\tau)+y(\tau)\Vert_U^2 - 2\re \langle u(\tau),y(\tau)\rangle_U \d\tau\\
    &= \int_0^t\Vert u(\tau) + y(\tau)\Vert^2_U\d\tau + \Vert x_0\Vert^2_X - \Vert x(t)\Vert_X^2,
\end{align*}
for every $t\geq 0$. Consequently, we have
\begin{align*}
    J(x_0,u) = \lim_{t\to\infty}\int_0^t\Vert u(\tau) + y(\tau)\Vert^2_U\d\tau + \Vert x_0\Vert^2_X - \lim_{t\to\infty}\Vert x(t)\Vert_X^2.
\end{align*}
Moreover, we look for an optimal input that is stabilizing, i.e., for which $\displaystyle\lim_{t\to\infty}\Vert x^{\opt}(t)\Vert_X = 0$ where $x^{\opt}$ is the state trajectory obtained by closing the loop with the optimal input. This implies that $J(x_0,u)$ is minimal if and only if $u(t) = -y(t)$, provided that it is stabilizing. This optimal input is called $u^{\opt}$. The corresponding state trajectory, denoted by $x^{\opt}(t)$, is given by $x^{\opt}(t) = \mathbb{T}^K_tx_0$.  Since $\mathbb{T}^K$ is strongly stable, $u^{\opt}$ is the unique stabilizing optimal control. In this case, $J(x_0,u^{\opt}) = \Vert x_0\Vert_X^2$, which implies that $\Pi = I$. Now we consider the operator node Riccati equation \eqref{ON_Riccati} with $\Pi = I$. It holds that 
\begin{align*}
&\left\langle A\& B\left[\begin{matrix}x\\ u\end{matrix}\right], \Pi x\right\rangle_X + \left\langle \Pi x, A\& B\left[\begin{matrix}x\\ u\end{matrix}\right]\right\rangle_X + \left\Vert C\& D\left[\begin{matrix}x\\ u\end{matrix}\right]\right\Vert^2_U + \left\Vert u\right\Vert^2_U\\
&= 2\re\left\langle A\& B\left[\begin{matrix}x\\ u\end{matrix}\right], x\right\rangle_X + \left\Vert C\& D\left[\begin{matrix}x\\ u\end{matrix}\right]\right\Vert^2_U + \left\Vert u\right\Vert^2_U.
\end{align*}
According to the impedance energy preserving property of $S$ characterized in Definition \ref{def:ImpPass}, it holds that $\re\langle A\&B\left[\begin{smallmatrix}x\\u\end{smallmatrix}\right],x\rangle = \re\langle C\&D\left[\begin{smallmatrix}x\\u\end{smallmatrix}\right],u\rangle$ for every $(x,u)^\top\in D(S)$. As a consequence, we have that
\begin{align*}
&\left\langle A\& B\left[\begin{matrix}x\\ u\end{matrix}\right], \Pi x\right\rangle_X + \left\langle \Pi x, A\& B\left[\begin{matrix}x\\ u\end{matrix}\right]\right\rangle_X + \left\Vert C\& D\left[\begin{matrix}x\\ u\end{matrix}\right]\right\Vert^2_U + \left\Vert u\right\Vert^2_U\\
&= 2\re\left\langle C\&D\left[\begin{matrix}x\\u\end{matrix}\right],u\right\rangle_U + \left\Vert C\& D\left[\begin{matrix}x\\ u\end{matrix}\right]\right\Vert^2_U + \left\Vert u\right\Vert^2_U\\
&= \left\Vert C\& D\left[\begin{matrix}x\\ u\end{matrix}\right] + u\right\Vert_U^2.
\end{align*}
Defining $E\&F:D(S)\to U$ by $E\&F\left[\begin{smallmatrix}x\\ u\end{smallmatrix}\right] = C\& D\left[\begin{smallmatrix}x\\ u\end{smallmatrix}\right] + u$ concludes the proof.
\end{proof}

\begin{remark}
The inequality $\langle x_0,\Pi x_0\rangle_X\leq \Vert x_0\Vert^2_X$ in Theorem \ref{thm:Thm_LQ_ImpPass} implies also that $\Vert \Pi\Vert_{\mathcal{L}(X)}\leq 1$. Indeed, since $\Pi = \Pi^*\in\mathcal{L}(X)$ is a positive operator, its square root is a well-defined bounded linear operator on $X$. Consequently, it holds that 
\eq{
\langle x_0,\Pi x_0\rangle_X = \Vert \Pi^{\frac{1}{2}}x_0\Vert^2_X\leq \Vert x_0\Vert^2_X,\qquad x_0\in X,
}
which implies that $\Vert \Pi^\frac{1}{2}\Vert_{\mathcal{L}(X)}\leq 1$, leading to $\Vert\Pi\Vert_{\mathcal{L}(X)}\leq 1$.

Moreover, note that the well-posedness of the system node $S$ is not required in Theorem \ref{thm:Thm_LQ_ImpPass}. Even in the case where $S$ is not well-posed, the interconnection of $S$ and the optimal control leads to a well-posed system.
\end{remark}

Now we consider the Popov function defined in \eqref{eq:Popov}, that is, the function $\Xi$ defined by $\Xi(i\omega) = I + P(i\omega)^*P(i\omega),\,\, \omega\in\mathbb{R}$.

The following theorem gives a simple factorization of the Popov function in the case of an impedance energy preserving system that exhibits mild conditions.

\begin{theorem}
Consider an impedance energy preserving system node $S := \left[\begin{smallmatrix}A\& B\\ C\&D\end{smallmatrix}\right]$ with transfer function $P$. We denote by $A$ its semigroup generator. Moreover, assume that $\sigma(A)\cap i\mathbb{R} = \emptyset$. Then, the Popov function defined in \eqref{eq:Popov} admits the factorization $\Xi(i\omega) = \chi(i\omega)^*\chi(i\omega), \omega\in\mathbb{R}$, where
\begin{equation*}
\chi(s) = P(s) + I, \qquad s\in\rho(A). 
\end{equation*}
Moreover, if $P\in\H^\infty(\mathcal{L}(U))$, then both $\chi$ and $\chi^{-1}$ are in $\H^\infty(\mathcal{L}(U))$, which makes $\chi$ a spectral factor of $\Xi$.
\end{theorem}
\begin{proof}
Take $\omega\in\mathbb{R}$. Because $i\omega\in\rho(A)$, the impedance energy preserving property of $S$ implies that $P(i\omega)^* + P(i\omega) = 0$, see e.g.~\citel{Sta02}{Thm.~4.6(iv)}. Hence, it holds that
\begin{align*}
\chi(i\omega)^*\chi(i\omega) &= (P(i\omega)^* + I)(P(i\omega)+I)\\
&= P(i\omega)^*P(i\omega) + P(i\omega)^* + P(i\omega) + I\\
&= \Xi(i\omega).
\end{align*}
To prove that $\chi$ is a spectral factor when $P\in\H^\infty(\mathcal{L}(U))$, assume that $P\in\H^\infty(\mathcal{L}(U))$. From this, it is easy to see that $\chi = P+I\in\H^\infty(\mathcal{L}(U))$. To prove that $\chi^{-1}\in\H^\infty(\mathcal{L}(U))$, note that because of impedance passivity, $P$ is a positive-real transfer function, see \citel{CurWei19}{Rem.~3.7}. Moreover, according to \citel{CurWei19}{Thm.~4.2}, $-I$ is an admissible feedback operator for $S$. In addition, the closed-loop transfer function, denoted by $\tilde{P}$, satisfies $\tilde{P} = P(I + P)^{-1}$, see \citel{CurWei06}{Sec.~2}. Furthermore, it has been shown in \citel{CurWei06}{Prop.~2.1} that $\tilde{P}\in\H^\infty(\mathcal{L}(U))$. The proof concludes by noting that $\chi^{-1} = (I+P)^{-1} = I-\tilde{P}\in\H^\infty(\mathcal{L}(U))$.
\end{proof}

As is mentioned in \cite{Opmeer_MTNS}, the function $\chi(s)$ may be computed as the transfer function of the system node $\left[\begin{smallmatrix}A\& B\\ E\& F\end{smallmatrix}\right]$. Indeed, by doing so, it holds that for $s\in\rho(A)$
\eq{
    \chi(s) = E\&F\left[\begin{matrix}(s-A)^{-1}B\\ I\end{matrix}\right] = C\&D\left[\begin{matrix}(s-A)^{-1}B\\ I\end{matrix}\right] + I = P(s) + I.
}
Moreover, if $S$ and its dual are both impedance energy preserving, that is, $S$ is impedance conservative, then the factorization of the Popov function extends to $\rho(A)$, i.e.
\begin{align*}
    \chi(-\overline{s})^*\chi(s) = I + P(-\overline{s})^*P(s), \qquad s\in\rho(A).
\end{align*}
This is a consequence of the identity $P(-\overline{s})^* + P(s) = 0$ for all $s\in\rho(A)$, see \citel{Sta02}{Thm.~4.7(iv)}.

\subsection{Scattering passive systems}
In this part, we consider a scattering passive system node $S := \left[\begin{smallmatrix}A\& B\\ C\&D\end{smallmatrix}\right]$ on the spaces $(U,X,Y)$. We denote by $\mathbb{T}$ its associated semigroup. We emphasize that the scattering passivity of $S$ makes it well-posed. The following theorem presents some properties of the optimal control that minimizes the cost \eqref{eq:Cost} in such a setting. Information about the optimal cost operator are also provided.

\begin{theorem}\label{thm:Thm_LQ_ScattPass}
Consider $S:=\left[\begin{smallmatrix}A\& B\\ C\&D\end{smallmatrix}\right]$ a scattering passive system node and consider the cost \eqref{eq:Cost} associated to the LQ optimal control problem. Then the finite cost condition holds and the optimal cost operator, denoted by $\Pi$, satisfies $\Pi\leq I$. If in addition $S$ is scattering energy preserving and if the semigroup $\mathbb{T}$ is strongly stable, then the unique stabilizing optimal control is $u(t) = 0, t\geq 0$. In this case, the optimal cost operator is given by $\Pi = I$. In addition, it satisfies the operator node Riccati equation with $E\&F:D(S)\to U, E\&F\left[\begin{smallmatrix}x\\ u\end{smallmatrix}\right] = \sqrt{2}u$.
\end{theorem}

\begin{proof}
According to the scattering passivity of $S$, it holds that 
\begin{align*}
    \int_0^t\Vert u(\tau)\Vert^2_U + \Vert y(\tau)\Vert^2_Y\d\tau &\leq 2\int_0^t\Vert u(\tau)\Vert^2_U\d\tau + \Vert x_0\Vert^2_X - \Vert x(t)\Vert^2_X\\ 
    &\leq 2\int_0^t\Vert u(\tau)\Vert^2_U\d\tau + \Vert x_0\Vert^2_X,
\end{align*}
along generalized solutions $(x,u,y)$ of $S$. Consequently,
the cost \eqref{eq:Cost} may be upper bounded as follows
\begin{align*}
    J(x_0,u)&\leq 2\lim_{t\to\infty}\int_0^t\Vert u(\tau)\Vert^2_U\d\tau + \Vert x_0\Vert^2_X - \lim_{t\to\infty}\Vert x(t)\Vert^2_X\\
    &\leq 2\lim_{t\to\infty}\int_0^t \Vert u(\tau)\Vert^2_U\d\tau + \Vert x_0\Vert_X^2.
\end{align*}
It is easy to see that, by considering the particular input $u(t) = 0, t\geq 0$, we have that 
\begin{align*}
    \langle x_0,\Pi x_0\rangle = J(x_0,u^{\opt})\leq J(x_0,u)\leq\Vert x_0\Vert_X^2,
\end{align*}
which implies both the finite cost condition and $\Pi\leq I$. Now assume that $S$ is scattering energy preserving and that $\mathbb{T}$ is strongly stable. Then, the cost \eqref{eq:Cost} satisfies 
\begin{align*}
    J(x_0,u)&= 2\lim_{t\to\infty}\int_0^t\Vert u(\tau)\Vert^2_U\d\tau + \Vert x_0\Vert^2_X - \lim_{t\to\infty}\Vert x(t)\Vert^2_X.
\end{align*}
Because we are looking for an optimal input that is stabilizing, the previous expression is minimal if and only if $u(t) = 0, t\geq 0$, provided that it is stabilizing. The state trajectory that corresponds to the input $u(t) = 0$ is given by $x^{\opt}(t) = \mathbb{T}_tx_0, t\geq 0$. Thanks to the strong stability of $\mathbb{T}$, $u(t) = 0$ is the unique stabilizing optimal control. If we call this optimal input $u^{\opt}$, it holds that $J(x_0,u^{\opt}) = \Vert x_0\Vert_X^2$ and the optimal cost operator is $\Pi = I$. According to Definition \ref{def:ScattPass}, the scattering energy preserving property of $S$ is equivalent to $2\re \left\langle A\& B
        \pmatsmall{x\\ u}
,x\right\rangle_X = \Vert u\Vert_U^2 - \left\Vert C\& D\pmatsmall{x\\ u }\right\Vert_Y^2$ for all $(x,u)^\top\in D(S)$. Consequently, it holds that 
\begin{align*}
&\left\langle A\& B\left[\begin{matrix}x\\ u\end{matrix}\right], \Pi x\right\rangle_X + \left\langle \Pi x, A\& B\left[\begin{matrix}x\\ u\end{matrix}\right]\right\rangle_X + \left\Vert C\& D\left[\begin{matrix}x\\ u\end{matrix}\right]\right\Vert^2_Y + \left\Vert u\right\Vert^2_U\\
&= 2\re\left\langle A\& B\left[\begin{matrix}x\\ u\end{matrix}\right], x\right\rangle_X + \left\Vert C\& D\left[\begin{matrix}x\\ u\end{matrix}\right]\right\Vert^2_Y + \left\Vert u\right\Vert^2_U\\
&= 2\Vert u\Vert^2_U.
\end{align*}
Defining $E\&F:D(s)\to U$ by $E\&F\pmatsmall{x\\u} = \sqrt{2}u$ concludes the proof.
\end{proof}

The following theorem gives a factorization of the Popov function for scattering energy preserving systems.

\begin{theorem}
Consider a scattering energy preserving system node $S := \left[\begin{smallmatrix}A\& B\\ C\&D\end{smallmatrix}\right]$ with transfer function $P$ and semigroup generator $A$. Moreover, assume that $\sigma(A)\cap i\mathbb{R} = \emptyset$. Then, the Popov function defined in \eqref{eq:Popov} admits the factorization $\Xi(i\omega) = \chi(i\omega)^*\chi(i\omega), \omega\in\mathbb{R}$, where $\chi$ is a spectral factor given by
\begin{equation*}
\chi(s) = \sqrt{2}I, \qquad s\in\rho(A). 
\end{equation*}
\end{theorem}
\begin{proof}
Take $\omega\in\mathbb{R}$. Because $i\omega\in\rho(A)$, the scattering energy preserving property of $S$ implies that $P(i\omega)^*P(i\omega) = I$, see e.g.~\citel{Sta02}{Thm.~3.4(iv)}. Hence, it holds that
\begin{align*}
\chi(i\omega)^*\chi(i\omega) &= 2I =  P(i\omega)^*P(i\omega) + I = \Xi(i\omega).
\end{align*}
It is easy to see that $\chi\in\H^\infty(\mathcal{L}(U))$ and $\chi^{-1}\in\H^\infty(\mathcal{L}(U))$, which makes $\chi$ a spectral factor of $\Xi$.
\end{proof}
Again, $\chi(s)$ may be obtained as the transfer function of $\pmatsmall{A\&B\\E\&F}$. Indeed, it is given by $E\&F\pmatsmall{(s-A)^{-1}B\\I} = \sqrt{2} I$. 

\section{Boundary control systems}\label{sec:LQ_BCS}

In this section we show how the results of Section \ref{sec:LQ_SystemNodes} translate to abstract boundary control systems of the form
\begin{subequations}
\label{eq:BCSlin}
\eqn{
\dot{x}(t)&= L x(t)
\\
G x(t) &=  u(t)\\
y(t) &= K x(t)
}
\end{subequations}
for $t\geq 0$. Those systems have been considered in e.g.~~\cite{Sal87a,MalSta06,JacobZwart}. If $X$, $U$, and $Y$ are Hilbert spaces, then
the triple $(G,L,K)$ is called a \emph{boundary node} on the spaces $(U,X,Y)$ if
$L: \Dom(L)\subset X\to X$, $G\in \Lin(\Dom(L),U)$,  $K\in \Lin(\Dom(L),Y)$ and the following hold.
\begin{itemize}
\item The restriction $A:=L\vert_{\ker(G)}$ with domain $\Dom(A)=\ker(G)$ generates a strongly continuous semigroup $\T$ on $X$.
\item The operator $G\in \Lin(\Dom(L),U)$ has a bounded right inverse, i.e., there exists $G^r\in\Lin(U,\Dom(L))$ such that $GG^r = I$.
\end{itemize}
This boundary node is impedance passive if $Y=U$ and if
\eq{
\re \iprod{L x}{x}_X \leq \re \iprod{G x}{K x}_U , \qquad x\in \Dom(L).
}
It is impedance energy preserving if we have equality in the above expression. Moreover, a boundary node $(G,L,K)$ is called scattering passive if 
\eq{
2\re \iprod{L x}{x}_X \leq \norm{Gx}^2_U - \norm{Kx}^2_Y, \qquad x\in \Dom(L),
}
while it is scattering energy preserving if the above inequality is an equality. The transfer function $P:\rho(A)\to \Lin(U,Y)$ of $(G,L,K)$ is defined so that $P(s)u = K x$, where $x\in \Dom(L)$ is the unique solution of the equations $(s-L)x=0$ and $Gx = u$.

It has been shown in~\cite{MalSta06} that to every boundary node $(G,L,K)$ we can associate a system node $S$ on $(U,X,Y)$. More precisely,
\citel{MalSta06}{Thm.~2.3} shows that
the operator
\eqn{
\label{eq:SysNode_BdNode}
S= \pmat{L\\ K}\pmat{I\\ G}\inv =: \pmat{A\&B\\C\&D}, \qquad \Dom(S)= \ran \left( \pmat{I\\ G} \right)
}
is a system node on $(U,X,Y)$. The definition of $S$ in particular implies that $(x,u)^\top\in \Dom(S)$ if and only if $x\in \Dom(L)$ and $u=Gx$.

In addition, the transfer function of $(G,L,K)$ coincides with the transfer function of the associated system node. Due to this connection with system nodes, a boundary node $(G,L,K)$ is (externally) well-posed if its associated system node $S$ is well posed, that is, if $S$ is a system node of a well-posed linear system $\Sigma$. If $x_0\in \Dom(L)$ and $u\in C^2(\zinf;U)$ are such that $G x_0=u(0)$, then by~\citel{MalSta06}{Lem.~2.6} there exist $x\in C^1(\zinf;X)\cap C(\zinf;\Dom(L))$ and $y\in C(\zinf;Y)$ such that $x(0)=x_0$ and such that~\eqref{eq:BCSlin} hold for all $t\geq 0$.
The well-posedness of $(G,L,K)$ is equivalent to the property that there exists $t,M_t>0$ such that all such solutions of~\eqref{eq:BCSlin} satisfy~\citel{JacobZwart}{Ch.~13}
\eq{
\norm{x(t)}_X^2 + \int_0^t\norm{y(\tau)}_Y^2\d \tau \leq M_t \left( \norm{x_0}_X^2 + \int_0^t \norm{u(\tau)}_U^2 \d\tau \right).
}

Now we show how the operator node Riccati equation \eqref{ON_Riccati} looks like in the context of boundary nodes. For this we consider a boundary node $(G,L,K)$ on the spaces $(U,X,Y)$ and its corresponding systems node given in \eqref{eq:SysNode_BdNode}. Now we take $(x,u)^\top\in D(S)$ and consider the operator node Riccati equation \eqref{ON_Riccati}. Observe that
\eq{
\left\langle A\&B\pmat{x\\u},\Pi x\right\rangle_X &+ \left\langle\Pi x,A\&B\pmat{x\\u}\right\rangle_X\\
&= \left\langle L \pmat{I\\ G}\inv \pmat{x\\u}, \Pi x\right\rangle_X + \left\langle\Pi x, L \pmat{I\\ G}\inv \pmat{x\\u}\right\rangle_X\\
&= \left\langle Lx,\Pi x\right\rangle_X + \left\langle \Pi x, Lx\right\rangle_X.
}
In addition, it holds that $u = Gx, C\&D\pmatsmall{x\\u} = K\pmatsmall{I\\ G}\inv \pmatsmall{x\\u} = Kx$ and $E\&F\pmatsmall{x\\u} = E\&F\pmatsmall{I\\ G}x =: Nx$. Consequently, the operator Riccati equation that is adapted for the LQ optimal control problem for the boundary node $(G,L,K)$ and that is obtained from \eqref{ON_Riccati} is given by 
\eqn{
\label{eq:Riccati_BndNode}
\langle Lx,\Pi x\rangle_X + \langle\Pi x,Lx\rangle_X + \Vert Gx\Vert_U^2 + \Vert Kx\Vert_Y^2 = \Vert Nx\Vert_U^2,
}
$x\in D(L)$. In the above equation, the unknowns are $\Pi = \Pi^*\in\mathcal{L}(X)$ and $N:D(L)\to U$.

We end this section by presenting two theorems that are direct adaptations of Theorems \ref{thm:Thm_LQ_ImpPass} and \ref{thm:Thm_LQ_ScattPass}.

\begin{theorem}\label{thm:ImpPass_LQ_BND}
Consider an impedance passive boundary node $(G,L,K)$ on the spaces $(U,X,U)$. The finite cost condition associated to the LQ optimal control problem for $(G,L,K)$ is satisfied and the optimal cost operator $\Pi$ satisfies $\Pi\leq I$. Moreover, if $(G,L,K)$ is impedance energy preserving and if the semigroup generated by $L_{\vert\ker(G+K)}$ is strongly stable, then $\Pi = I$. In addition, the unique optimal control is given by $u^{\opt}(t) = -Kx^{\opt}(t), t\geq 0$, where $x^{\opt}(t)$ is subject to $\dot{x}^{\opt}(t) = Lx^{\opt}(t), (G+K)x^{\opt}(t) = 0$. Furthermore, $\Pi = I$ satisfies the boundary node Riccati equation \eqref{eq:Riccati_BndNode} with $N:D(L)\to U, Nx = (G+K)x$.
\end{theorem}

\begin{proof}
We start by noting that the operator $A^K$ in \eqref{eq:StabAK} is exactly $L_{\vert\ker(G+K)}$. The proof follows by Theorem \ref{thm:Thm_LQ_ImpPass} by noting that the system node $S$ associated to the boundary node $(G,L,K)$ is given in \eqref{eq:SysNode_BdNode} and that $(x,u)^\top\in D(S)$ if and only if $x\in D(L)$ and $u=Gx$.
\end{proof}

\begin{theorem}\label{thm:ScattPass_LQ_BND}
Let $(G,L,K)$ be a scattering passive boundary node on the spaces $(U,X,Y)$. The finite cost condition associated to the LQ optimal control problem for $(G,L,K)$ is satisfied and the optimal cost operator $\Pi$ satisfies $\Pi\leq I$. In addition, if $(G,L,K)$ is scattering energy preserving and if the semigroup generated by $L_{\vert\ker(G)}$ is strongly stable, then $\Pi = I$ and the optimal control is given by $u^{\opt}(t) = 0$, $t\geq 0$. Moreover, $\Pi = I$ satisfies the boundary node Riccati equation \eqref{eq:Riccati_BndNode} with $N:D(L)\to U, Nx = \sqrt{2}\,Gx$.
\end{theorem}

\begin{proof}
The proof is a direct consequence of Theorem \ref{thm:Thm_LQ_ScattPass} where we note that $A = L_{\vert\ker(G)}$.
\end{proof}

\section{A class of infinite-dimensional port-Hamiltonian systems}\label{sec:App_pHs}
In this section we apply our results to the optimal control of \emph{first-order infinite-dimensional port-Hamiltonian systems}~\cite{JacobZwart,JacobMorrisZwart,Aug19}.
This class consists of partial differential equations of the form
\begin{subequations}
\label{eq:PHSmain}
\begin{align}
\frac{\partial x}{\partial t}(\zeta,t) &= P_1\frac{\partial}{\partial \zeta}\left(\mathcal{H}(\zeta)x(\zeta,t)\right) + P_0\left(\mathcal{H}(\zeta)x(\zeta,t)\right)
\label{pH}\\
\left[\begin{matrix}I\\ 0\end{matrix}\right]u(t)  &= \left[\begin{matrix} \WBi\\ \WBh\end{matrix}\right]\left[\begin{matrix}
\mathcal{H}(b)x(b,t)\\
\mathcal{H}(a)x(a,t) 
\end{matrix}\right]
\label{Input_pH}
\\
y(t) &= \WCo \left[\begin{matrix}
\mathcal{H}(b)x(b,t)\\
\mathcal{H}(a)x(a,t) 
\end{matrix}
\right]
\label{Output_pH}
\end{align}
\end{subequations}
for $t\geq 0$ and $\zeta\in [a,b]$, where $x(\zeta,t)\in \Cc^n$.
Here $P_1 = P_1^*\in \Cc^{n\times n}$ is assumed to be invertible and $P_0\in \Cc^{n\times n}$. The function $\mathcal{H}(\cdot)\in\L^\infty(a,b;\mathbb{C}^{n\times n})$ is assumed to satisfy $\mathcal{H}(\zeta)^* = \mathcal{H}(\zeta)$ and $ \mathcal{H}(\zeta)\ge cI$ for some $c>0$ and for a.e. $\zeta\in [a,b]$.
The matrices $\tilde W_{B,1}$, $\tilde W_{B,2}$, and $\tilde W_C$ determine the boundary input $u(t)\in \Cc^p$, homogeneous boundary conditions, and boundary output $y(t)\in\Cc^k$.
As shown in~\cite{JacobZwart}, the class of port-Hamiltonian systems~\eqref{eq:PHSmain} comprises mathematical models of mechanical systems such as one-dimensional wave and Timoshenko beam equations. In what follows, we denote
\eq{
R_0 = \frac{1}{\sqrt{2}}\pmat{P_1&-P_1\\I&I}, \qquad \Xi = \pmat{0&I\\I&0}, 
\qquad 
 W_B=\pmat{\tilde W_{B,1}\\ \tilde W_{B,2}}R_0\inv,
}
$W_{B,1} = \tilde W_{B,1} R_0\inv$,
$W_{B,2} = \tilde W_{B,2} R_0\inv$,
and $W_C = \tilde W_C R_0\inv$. Moreover, because $P_1\mathcal{H}$ is Hermitian, it may be diagonalize in the following way
\begin{equation}
P_1 \mc H(\zeta)= S(\zeta)\inv \pmat{\Lambda(\zeta) & 0\\0 & \Theta(\zeta)}S(\zeta),\qquad\zeta\in(a,b),
\label{eq:diag_P1H}
\end{equation}
where $\Lambda$ and $\Theta$ are diagonal matrices whose elements are the positive and negative eigenvalues of $P_1\mathcal{H}$, respectively. We denote by $Z^+(\zeta)$ the span of eigenvectors of $P_1\mathcal{H}$ corresponding to positive eigenvalues and $Z^-(\zeta)$ the span of eigenvectors of $P_1\mathcal{H}$ corresponding to negative eigenvalues. Moreover, we consider the decomposition $\tilde{W}_B := \pmatsmall{\tilde{W}_b & \tilde{W}_a}, \tilde{W}_b, \tilde{W}_a\in\mathbb{C}^{n\times n}$. The linear operator that is associated to the homogeneous part of \eqref{eq:PHSmain} is denoted by $A$ and is given by 
\begin{subequations}
\begin{align}
Ax &= P_1\frac{\d}{\d\zeta}(\mathcal{H}x) + P_0(\mathcal{H}x),\label{A_PHS}\\
    D(A) &= \{x\in X\,\bigm|\, \mathcal{H}x\in\H^1(a,b;\mathbb{C}^n),  \tilde{W}_B\left[\begin{smallmatrix}(\mathcal{H}x)(b)\\(\mathcal{H}x)(a)\end{smallmatrix}\right] = 0\},\label{D(A)_PHS}
\end{align}
\end{subequations}
where the state space $X := \L^2(a,b;\mathbb{C}^n)$ is equipped with the inner product $\langle f,g\rangle_X := \langle f,\mathcal{H}g\rangle_{\L^2(a,b;\mathbb{C}^n)}$.
We make the following assumption for the remaining of this section.
\begin{assumption}
\label{ass:PHSass}
We have $\rank \left(\tilde{W_B}\right)=n$. Moreover, $S, \Lambda$ and $\Theta$ in \eqref{eq:diag_P1H} are continuously differentiable. In addition, it holds that
\begin{equation*}
\tilde{W}_b\mathcal{H}(b)Z^+(b)\oplus \tilde{W}_a\mathcal{H}(a)Z^-(a) = \C^n.
\end{equation*}
\end{assumption}
\begin{remark}
According to \citel{JacobMorrisZwart}{Thm.~1.5},
the conditions in Assumption \ref{ass:PHSass} are equivalent for the linear operator $A$, see \eqref{A_PHS}--\eqref{D(A)_PHS}, to be the generator of a strongly continuous semigroup on $X$. Hence, this allows to recast \eqref{eq:PHSmain} as a boundary control system, see~\citel{JacobZwart}{Thm.~11.3.2}. We emphasize that Assumption \ref{ass:PHSass} does not imply that \eqref{eq:PHSmain} describes a well-posed linear system in the sense of Definition \ref{def:WP}. This would require an additional assumption, as described in \citel{JacobZwart}{Thm.~13.2.2} 
\end{remark}

As shown in~\citel{JacobZwart}{Ch.~11--13},~\eqref{eq:PHSmain} can be recast as a boundary control system, see also~\cite{VilZwa05,PhiRei25}. The boundary node $(G,L,K)$ on $(\Cc^p,X,\Cc^k)$ associated to the system~\eqref{eq:PHSmain} is defined by 
\begin{subequations}
\label{eq:PHSBCS}
\begin{align}
    Lx &= P_1\frac{\d}{\d\zeta}(\mathcal{H}x) + P_0(\mathcal{H}x),\\
    D(L) &= \{x\in X\,\bigm|\, \mathcal{H}x\in\H^1(a,b;\mathbb{C}^n),\  \WBh\left[\begin{smallmatrix}(\mathcal{H}x)(b)\\(\mathcal{H}x)(a)\end{smallmatrix}\right] = 0\},\label{D(L)_PHS}\\
    Gx &= \WBi\left[\begin{smallmatrix}(\mathcal{H}x)(b)\\(\mathcal{H}x)(a)\end{smallmatrix}\right], \quad
    Kx = \WCo\left[\begin{smallmatrix}(\mathcal{H}x)(b)\\(\mathcal{H}x)(a)\end{smallmatrix}\right].
\end{align}
\end{subequations}
Assumption~\ref{ass:PHSass} implies that $(G,L,K)$ is a boundary node on $(\Cc^p,X,\Cc^k)$~\citel{JacobZwart}{Thm.~11.3.2 \& Rem.~11.3.3}.

The following lemma gives equivalent conditions under which the boundary node $(G,L,K)$ is impedance passive (energy preserving) or scattering passive (energy preserving). 

\begin{lemma}\label{lemma_ImpPass_ScattPass_PHS}
The boundary node $(G,L,K)$ defined in \eqref{eq:PHSBCS} is impedance passive if and only if $k=p$ and
\eqn{
\label{ImpPass_PHS}
2\re \iprod{  W_{B,1} g}{W_C g}_{\C^p}\ge \iprod{\Xi g}{g}_{\C^{2n}},
}
for all $g\in \ker(W_{B,2})$. It is impedance energy preserving if and only if \eqref{ImpPass_PHS} holds with an equality.
In addition, $(G,L,K)$ is scattering passive if and only if
\eqn{
\label{ScattPass_PHS}
\Vert W_{B,1} g\Vert^2_{\C^p} - \Vert W_C g\Vert^2_{\C^k}\ge \iprod{\Xi g}{g}_{\C^{2n}},
}
for all $g\in \ker(W_{B,2})$. It is scattering energy preserving if and only if \eqref{ScattPass_PHS} holds as an equality.
\end{lemma}

\begin{proof}
Take $x\in \Dom(L)$. Then, $\mc H x\in \H^1(a,b;\Cc^n)$. Now we define $g=R_0 \pmatsmall{(\mc Hx)(b)\\ (\mc Hx)(a)}$. Consequently $W_{B,2}g=0$, $Gx = W_{B,1}g$, and $Kx=W_C g$. We have from~\citel{JacobZwart}{Lem.~7.2.1 \& (7.26)} that $\re \iprod{Lx}{x}_X=
\frac{1}{2} \iprod{\Xi g}{g}_{\C^{2n}}$. The proof concludes by noting that $\re \iprod{W_{B,1} g}{W_C g}_{\C^p}
=   \re \iprod{Gx}{Kx}_{\C^p}$ when $p=k$ and $\Vert Gx\Vert_{\C^p} = \Vert W_{B,1} g\Vert_{\C^p}$, $\Vert Kx\Vert_{\C^k} = \Vert W_C g\Vert_{\C^k}$.
\end{proof}

We end this section by presenting two main theorems for the LQ optimal control of port-Hamiltonian systems. Those theorems are adaptations of Theorems \ref{thm:Thm_LQ_ImpPass} and \ref{thm:Thm_LQ_ScattPass} recast for \eqref{eq:PHSmain}.

\begin{theorem}\label{thm:ImpPass_PHS}
Consider the boundary node $(G,L,K)$ as in \eqref{eq:PHSBCS} on the spaces $(\C^p,X,\C^p)$. Assume moreover that $2\re \iprod{  W_{B,1} g}{W_C g}_{\C^p}\ge \iprod{\Xi g}{g}_{\C^{2n}}$ holds for all $g\in\ker(W_{B,2})$. Then, the finite cost condition associated to the LQ optimal control problem for $(G,L,K)$ is satisfied and the optimal cost operator $\Pi$ satisfies $\Pi\leq I$. Moreover, if $2\re \iprod{  W_{B,1} g}{W_C g}_{\C^p}= \iprod{\Xi g}{g}_{\C^{2n}}$ for all $g\in\ker(W_{B,2})$ and if the partial differential equations \eqref{pH} together with the boundary conditions 
\eqn{
\label{BC_PHS_OutputFeedback}
\pmat{\tilde{W}_{B,1}+\tilde{W}_C\\\tilde{W}_{B,2}}\pmat{\mathcal{H}(b)x(b,t)\\\mathcal{H}(a)x(a,t)} = 0
}
is strongly stable, then $\Pi = I$. In addition, the unique optimal control is given by $u^{\opt}(t) = -\tilde{W}_C\pmatsmall{\mathcal{H}(b)x^{\opt}(b,t)\\\mathcal{H}(a)x^{\opt}(a,t)}, t\geq 0$, where $x^{\opt}(t)$ is subject to the PDEs \eqref{pH} with the boundary conditions \eqref{BC_PHS_OutputFeedback}. Furthermore, $\Pi = I$ satisfies the boundary node Riccati equation \eqref{eq:Riccati_BndNode} with $N:D(L)\to \C^p, Nx = \pmatsmall{\tilde{W}_{B,1} + \tilde{W}_C}\pmatsmall{\mathcal{H}(b)x(b)\\\mathcal{H}(a)x(a)}$.
\end{theorem}

\begin{proof}
According to Lemma \ref{lemma_ImpPass_ScattPass_PHS}, \eqref{ImpPass_PHS} is equivalent to $(G,L,K)$ being impedance passive. Moreover, the classical solutions of \eqref{pH} with \eqref{BC_PHS_OutputFeedback} are $\T^{\ker(G+K)}_tx_0, t\geq 0$, where $\T^{\ker(G+K)}$ is the semigroup generated by $L_{\vert\ker(G+K)}$ and $x_0$ is the initial condition. The proof then follows by Theorem \ref{thm:ImpPass_LQ_BND}.
\end{proof}

\begin{remark}
In the case where $p = n$, it has been shown in \cite[Lem.~7]{PauLeG21} that any negative output feedback $u(t) = -\kappa y(t), \kappa > 0$, is exponentially stabilizing for impedance passive port-Hamiltonian systems described by \eqref{eq:PHSmain}. Under the additional assumption that such systems are impedance energy preserving, this in particular implies that the unique optimal control is given by $u(t) = -y(t)$ according to Theorem \ref{thm:ImpPass_PHS}. As shown in \cite{PauLeG21}, this result also applies to impedance passive port-Hamiltonian systems of arbitrary order. Contrary to \cite{Pau19}, Theorem \ref{thm:ImpPass_PHS} has not the assumption that $p=n$. However, strong stability is not for free and has to be assumed.
\end{remark}

\begin{theorem}
Consider the boundary node $(G,L,K)$ as in \eqref{eq:PHSBCS} on the spaces $(\C^p,X,\C^k)$. Assume moreover that $\Vert W_{B,1} g\Vert_{\C^p}^2 - \Vert W_C g\Vert^2_{\C^k}\ge \iprod{\Xi g}{g}_{\C^{2n}}$ holds for all $g\in\ker(W_{B,2})$. Then, the finite cost condition associated to the LQ optimal control problem for $(G,L,K)$ is satisfied and the optimal cost operator $\Pi$ satisfies $\Pi\leq I$. Moreover, if $\Vert W_{B,1} g\Vert_{\C^p}^2 - \Vert W_C g\Vert^2_{\C^k} = \iprod{\Xi g}{g}_{\C^{2n}}$ holds for all $g\in\ker(W_{B,2})$ and if the partial differential equations \eqref{pH} together with the boundary conditions 
\eqn{
\label{BC_PHS}
\pmat{\tilde{W}_{B,1}\\\tilde{W}_{B,2}}\pmat{\mathcal{H}(b)x(b,t)\\\mathcal{H}(a)x(a,t)} = 0
}
is strongly stable, then $\Pi = I$. In addition, the unique optimal control is given by $u^{\opt}(t) = 0, t\geq 0$. Furthermore, $\Pi = I$ satisfies the boundary node Riccati equation \eqref{eq:Riccati_BndNode} with $N:D(L)\to \C^p, Nx = \sqrt{2}\tilde{W}_{B,1}\pmatsmall{\mathcal{H}(b)x(b)\\\mathcal{H}(a)x(a)}$.
\end{theorem}

\begin{proof}
According to Lemma \ref{lemma_ImpPass_ScattPass_PHS}, \eqref{ScattPass_PHS} is equivalent to $(G,L,K)$ being scattering passive. Moreover, the classical solutions of \eqref{pH} with \eqref{BC_PHS} coincides with $\T^{\ker(G)}_tx_0$, where $\T^{\ker(G)}$ is the semigroup generated by $L_{\vert\ker(G)}$ and $x_0$ is the initial condition. The proof then follows by Theorem \ref{thm:ScattPass_LQ_BND}.
\end{proof}

\section{An Euler-Bernoulli beam with shear force control}\label{sec:EB}

We consider the LQ optimal control problem for an Euler-Bernoulli beam
\begin{subequations}
\label{eq:Euler-Bernoulli}
\begin{align}
&\frac{\partial^2 w}{\partial t^2}(\zeta,t) = -\frac{\partial^4 w}{\partial\zeta^4}(\zeta,t),\\
&w(0,t) = \frac{\partial w}{\partial\zeta}(0,t) = \frac{\partial^2w}{\partial\zeta^2}(1,t) = 0,\label{eq:BC_EB}\\
&u(t) = \frac{\partial^3 w}{\partial\zeta^3}(1,t),\label{eq:Input_EB}\\
&y(t) = -\frac{\partial w}{\partial t}(1,t) + \varepsilon\frac{\partial^3 w}{\partial\zeta^3}(1,t),\label{eq:Output_EB}
\end{align}
\end{subequations}
where $w(\zeta,t)$ is the vertical displacement of the beam at position $\zeta\in(0,1)$ and time $t\geq 0$. The boundary conditions \eqref{eq:BC_EB} mean that the beam is clamped at $\zeta = 0$ and that there is no bending moment at $\zeta=1$. The shear force at $\zeta = 1$ is controlled by $u(t)$ and the measurement $y(t)$ is the velocity at $\zeta = b$ perturbed by the term $\varepsilon\frac{\partial^3w}{\partial\zeta^3}(1,t)$, $\varepsilon>0$, which may be viewed as a perturbation induced by the input at the same boundary point.
Such a model has been considered a lot in the literature. For instance, it has been used as a simplified model for the \emph{mast control system} in NASA's COFS (Control of flexible structures), see e.g. \cite{CheKra87}. Moreover, systems like \eqref{eq:Euler-Bernoulli} have attracted attention concerning the Riesz-basis property, see e.g. \cite{XuWei11} and \cite{GuoWan05}. The control objective in this section is to minimize the cost \eqref{eq:Cost}
subject to the dynamics \eqref{eq:Euler-Bernoulli}. We start by defining the auxiliary variables $x_1(\zeta,t) := w(\zeta,t)$ and $x_2(\zeta,t) := \frac{\partial w}{\partial t}(\zeta,t)$. With these variables, \eqref{eq:Euler-Bernoulli} may be written in an abstract form as $\pmatsmall{\dot{x}(t)\\y(t)} = \pmatsmall{A\& B\\ C\&D}\pmatsmall{x(t)\\u(t)}$, where the system node $S := \pmatsmall{A\& B\\ C\&D}$ is given by
\begin{subequations}
\label{eq:EB_SN}
\begin{align}
&A\&B\pmat{x\\u} := \pmat{x_2\\-\frac{\d^4x_1}{\d\zeta^4}}\\
&C\&D\pmat{x\\ u} := -x_2(1) + \varepsilon \frac{\d^3 x_1}{\d\zeta^3}(1)\\
&D(S) = \left\{\pmat{x\\u}\in \left( X\cap\left(\H^4(0,1;\Cc)\times\H^2(0,1;\Cc)\right)\right)\times\Cc, \right.\\
&\hspace{1cm}x_2(0) = \frac{\d x_2}{\d\zeta}(0) = 0,\\
&\hspace{1cm}\left.\frac{\d^2 x_1}{\d\zeta^2}(1) = 0, \frac{\d^3 x_1}{\d\zeta^3}(1) = u\right\}
\end{align}
\end{subequations}
with $x := (x_1,x_2)^\top$ and 
where the Hilbert state space $X$ is defined by 
\begin{align*}
    X := \left\{x\in\H^2(0,1;\Cc)\times\L^2(0,1;\Cc), x_1(0) = \frac{\d x_1}{\d\zeta}(0) = 0\right\},
\end{align*}
equipped with the inner product
\begin{align*}
    \langle f,g\rangle_X := \int_0^1 \frac{\d^2 f_1}{\d\zeta^2}\overline{\frac{\d^2g_1}{\d\zeta^2}} + f_2\overline{g_2}\d\zeta, 
\end{align*}
$f := (f_1,f_2)^\top, g := (g_1,g_2)^\top$.
A simple computation shows that 
\eq{
\label{eq:ImpPass_EB_SN}
\re\iprod{A\&B\pmatsmall{x\\u}}{x}_X = \re\iprod{C\&D\pmatsmall{x\\u}}{u}_\Cc - \varepsilon\Vert u\Vert^2_\Cc,
}
for all $(x,u)^\top\in D(S)$, which implies that $S$ is impedance passive. Along the generalized solutions $(x,u,y)$ of $S$, the equality above is equivalent to
\begin{equation}
    \Vert x(t)\Vert^2_X - \Vert x_0\Vert_X^2 = 2\re \int_0^t\langle u(\tau),y(\tau)\rangle_\Cc - \varepsilon\Vert u(\tau)\Vert^2_\Cc\d\tau,
\label{eq:ImpPass_EB_Int}
\end{equation}
$t\geq 0$. The estimate \eqref{eq:ImpPass_EB_Int} actually reveals that $S$ is \emph{strictly input passive}, see e.g. \citel{Augner_PhD}{Chap.~5}. The notion of \emph{strictly input or output passive} systems has already been defined in \cite{vdS99_Book} for nonlinear systems. Few years later, this property has been shown useful for many reasons. In particular, it is shown in \cite{RamZwa13} and \cite{AugJac14} that finite-dimensional strictly input passive controllers are able to stabilize a significantly large class of boundary control systems. A strictly output passive system has also interesting properties. For instance, for strictly output passive infinite-dimensional first-order port-Hamiltonian systems, exponential stability of the system's semigroup is equivalent to the system being exactly controllable or exactly observable in finite time, see e.g.~\citel{Villegas}{Thm.~3.22}. Many other interesting properties of such systems may be found in \citel{Villegas}{Chap.~3}. The following proposition gives the expression of the optimal input that minimizes the cost \eqref{eq:Cost} under the dynamics \eqref{eq:Euler-Bernoulli}.
\begin{proposition}
The unique optimal input $u^{\opt}$ that minimizes the cost \eqref{eq:Cost} under the dynamics of the Euler-Bernoulli beam \eqref{eq:Euler-Bernoulli} is given by 
\eqn{
\label{eq:OptimalInput_EB}
u^{\opt}(t) &= -\mu\left[ \varepsilon\frac{\partial^3 w^{\opt}}{\partial\zeta^3}(1,t)-\frac{\partial w^{\opt}}{\partial t}(1,t)\right],
}
$\mu := \left(\sqrt{1+\varepsilon^2}-\varepsilon\right)^{-1}$, where $w^{\opt}$ satisfies the following partial differential equations
\begin{subequations}
\label{eq:Euler-Bernoulli_opt}
\begin{align}
&\frac{\partial^2 w^{\opt}}{\partial t^2}(\zeta,t) = -\frac{\partial^4 w^{\opt}}{\partial\zeta^4}(\zeta,t),\\
&w^{\opt}(0,t) = \frac{\partial w^{\opt}}{\partial\zeta}(0,t) = \frac{\partial^2w^{\opt}}{\partial\zeta^2}(1,t) = 0,\label{eq:BC_EB_opt}\\
&\frac{\partial w^{\opt}}{\partial t}(1,t) = \sqrt{1+\varepsilon^2}\frac{\partial^3w^{\opt}}{\partial\zeta^3}(1,t).\label{eq:New_BC_EB_opt}
\end{align}
\end{subequations}
Moreover, the optimal cost $J(w_0,u^{\opt})$ is given by
\eq{
J(w_0,u^{\opt}) = \mu^{-1}\Vert x_0\Vert_X,
}
with $x_0(\cdot) = (w^{\opt}(\zeta,0), \frac{\partial w^{\opt}}{\partial t}(\zeta,0))^\top, \zeta\in(0,1)$.
\end{proposition}
\begin{proof}
Let $\mu := \left(\sqrt{1+\varepsilon^2}-\varepsilon\right)^{-1}$. First observe that \eqref{eq:Euler-Bernoulli} may be written in an abstract form as $\pmatsmall{\dot{x}(t)\\y(t)} = \pmatsmall{A\& B\\ C\&D}\pmatsmall{x(t)\\u(t)}$, where the system node $S := \pmatsmall{A\&B\\ C\&D}$ is given in  \eqref{eq:EB_SN}. Moreover, $S$ is impedance passive according to \eqref{eq:ImpPass_EB_Int}. Hence, thanks to Theorem \ref{thm:Thm_LQ_ImpPass}, the finite cost condition is satisfied and the optimal cost operator, denoted by $\Pi$, satisfies $\Pi\leq I$. Now observe that, thanks to the estimate \eqref{eq:ImpPass_EB_Int}, the cost \eqref{eq:Cost} may be written as
\begin{align*}
J(x_0,u) &= \int_0^\infty \Vert u(t)\Vert^2_\Cc+\Vert y(t)\Vert^2_\Cc\d t\\
&= \int_0^\infty\Vert \mu^{-1}u(t) + y(t)\Vert^2_\Cc\d t + \mu^{-1}\left[\Vert x_0\Vert^2_X - \lim_{t\to\infty}\Vert x(t)\Vert^2_X\right].
\end{align*}
Because we look for a stabilizing optimal control, it is clear from the previous expression that $u(t) = -\mu y(t)$ is the unique optimal control provided that it is strongly stabilizing. We denote it by $u^{\opt}$. In this case, the optimal cost is given by $J(x_0,u^{\opt}) = \mu^{-1}\Vert x_0\Vert^2_X$, with $\Pi = \mu^{-1}I$ the optimal cost operator. In order to analyze the stability of the closed-loop system, observe that injecting $u(t) = -\mu y(t)$ into \eqref{eq:Euler-Bernoulli} yields a system that is described by \eqref{eq:Euler-Bernoulli_opt}. The latter may be written in state-space form as $\dot{x}^{\opt}(t) = A^{\opt}x^{\opt}(t)$, where 
\begin{subequations}
\label{eq:OpA_opt}
\begin{align*}
A^{\opt}x &= \pmatsmall{x_2\\-\frac{\d^4x_1}{\d\zeta^4}}\\
D(A^{\opt}) &= \left\{x\in X\cap\left(\H^4(0,1)\times\H^2(0,1)\right),\right.\\
&\hspace{1cm}x_2(0) = \frac{\d x_2}{\d\zeta}(0) = 0,\\
&\hspace{1cm}\left.\frac{\d^2 x_1}{\d\zeta^2}(1) = 0, \frac{\d^3 x_1}{\d\zeta^3}(1) = \alpha x_2(1)\right\},
\end{align*}
\end{subequations}
where $\alpha = \left(\sqrt{1+\varepsilon^2}\right)^{-1}$. According to \citel{ConMor98}{Lem.~3.1}, the operator $A^{\opt}$ is the infinitesimal generator of an exponentially stable strongly continuous semigroup for any $\alpha > 0$, which means that $u^{\opt}(t)$ is stabilizing for \eqref{eq:Euler-Bernoulli} and hence it is the unique stabilizing optimal control. Now, we concentrate on the operator node Riccati equation \eqref{ON_Riccati}. We shall check that \eqref{ON_Riccati} is satisfied with $\Pi = \mu^{-1} I$ and $E\&F\pmatsmall{x\\u} = C\&D\pmatsmall{x\\u} + \mu^{-1} u, (x,u)^\top\in D(S)$. Take $(x,u)^\top\in D(S)$. It holds that
\begin{align*}
&\iprod{A\&B\pmatsmall{x\\u}}{\Pi x}_X + \iprod{\Pi x}{A\&B\pmatsmall{x\\u}}_X + \Vert u\Vert^2_\C + \Vert C\&D\pmatsmall{x\\u}\Vert^2_\C\\
&= 2\mu^{-1}\re\iprod{A\&B\pmatsmall{x\\u}}{x}_X + \Vert u\Vert^2_\C + \Vert C\&D\pmatsmall{x\\u}\Vert^2_\C\\
&= 2\mu^{-1}\re\iprod{C\&D\pmatsmall{x\\u}}{u}_\C - 2\varepsilon\mu^{-1}\Vert u\Vert^2_\C + \Vert u\Vert^2_\C + \Vert C\&D\pmatsmall{x\\u}\Vert^2_\C\\
&= 2\mu^{-1}\re\iprod{C\&D\pmatsmall{x\\u}}{u}_\C + \mu^{-2}\Vert u\Vert^2_\C + \Vert C\&D\pmatsmall{x\\u}\Vert^2_\C\\
&= \Vert C\&D\pmatsmall{x\\u} + \mu^{-1} u\Vert^2_\C\\
&:= \Vert E\&F\pmatsmall{x\\u}\Vert_\C^2.
\end{align*}
According to \cite{Opmeer_MTNS}, $u^{\opt}$ is a solution to $E\&F\pmatsmall{x^{\opt}\\u^{\opt}} = 0$, which entails that 
$u^{\opt}(t) = -\mu C\&D\pmatsmall{x^{opt}\\u^{\opt}}$. This confirms what has been found earlier.
\end{proof}

\begin{remark}
It is easy to see that $0<\mu^{-1} = \sqrt{1+\varepsilon^2}-\varepsilon<1$, which implies that $\Pi\leq I$. This corroborates what has been shown in Theorems \ref{thm:Thm_LQ_ImpPass} and \ref{thm:ImpPass_LQ_BND}.
\end{remark}

\section*{Acknowledgments}
This work was supported by the German Research Foundation (DFG). A. H. is supported by the DFG under the Grant HA 10262/2-1. The authors thank Hans Zwart for helpful discussions on the paper.

\bibliographystyle{elsarticle-num} 
\bibliography{biblio}

\end{document}